\documentclass[12pt,oneside]{amsart}

\usepackage[pdfencoding=auto]{hyperref}
\usepackage[utf8]{inputenc}
\usepackage[scale=0.7]{geometry}
\usepackage{microtype}
\usepackage{graphicx}
\usepackage{xcolor}
\usepackage{url}

\title{Zero-sum squares in bounded discrepancy $\mathbf{\{-1,1\}}$-matrices}

\author{Alma R. Arévalo}
\address[A. R. Arévalo]{Instituto de Matemáticas, UNAM}
\email{arevalo@ciencias.unam.mx}

\author{Amanda Montejano}
\address[A. Montejano]{Facultad de Ciencias, UNAM campus Juriquilla}
\email{amandamontejano@ciencias.unam.mx}

\author{Edgardo Roldán-Pensado}
\address[E. Roldán-Pensado]{Centro de Ciencias Matemáticas, UNAM Campus Morelia}
\email{e.roldan@im.unam.mx}

\newtheorem{teo}{Theorem}
\newtheorem{lema}[teo]{Lemma}
\newtheorem{coro}[teo]{Corollary}
\newtheorem*{obs}{Observation}
\newtheorem{conj}[teo]{Conjecture}
\newtheorem{claim}{Claim}
\theoremstyle{remark}

\newcommand{\abs}[1]{\left\lvert #1 \right\rvert}
\newcommand{\floor}[1]{\left\lfloor #1 \right\rfloor}
\newcommand{\ceil}[1]{\left\lceil #1 \right\rceil}
\newcommand{\M}{\mathcal M}
\newcommand{\N}{\mathcal N}
\DeclareMathOperator{\disc}{disc}

\begin{document}

\begin{abstract}
	For $n\ge 5$, we prove that every $n\times n$ matrix $\M=(a_{i,j})$ with entries in $\{-1,1\}$ and absolute discrepancy $\abs{\disc(\M)}=\abs{\sum a_{i,j}}\le n$ contains a zero-sum square except for the split matrix (up to symmetries). Here, a square is a $2\times 2$ sub-matrix of $\M$ with entries $a_{i,j}, a_{i+s,s}, a_{i,j+s}, a_{i+s,j+s}$ for some $s\ge 1$, and a split matrix is a matrix with all entries above the diagonal equal to $-1$ and all remaining entries equal to $1$.
	In particular, we show that for $n\ge 5$ every zero-sum $n\times n$ matrix with entries in $\{-1,1\}$ contains a zero-sum square.
\end{abstract}

\maketitle

\section{Introduction}
An \emph{Erickson matrix} is a square binary matrix that contains no \emph{squares} (defined below) with constant entries. 
In \cite{Eri1996}, Erickson asked for the maximum value of $n$ for which there exists an $n\times n$ Erickson matrix.
In \cite{AM2008} Axenovich and Manske gave an upper bound of around $2^{2^{40}}$. This gargantuan bound was later improved by Bacher and Eliahou in \cite{BE2010} using computational means to the optimal value of $15$.

This paper is devoted to studying a zero-sum analogue of Erickson matrices considering matrices with entries in $\{-1,1\}$. For this purpose, of course, we need to take into account the discrepancy or deviation of the matrix.

Discrepancy theory is an important branch in combinatorics with deep connections to many other areas in mathematics (see \cite{Cha2001} for a good general reference on this topic). In particular, one important result is Tao's recent proof of the Erdős discrepancy conjecture, \cite{Tao2016}, which states that any sequence of the form $f:\mathbb N\to \{-1,1\}$ satisfies that $\sup_{n,d}\abs{\sum_{j=1}^n f(jd)}=\infty$.

In recent years, we have witnessed the study of zero-sum structures becoming increasingly popular. Some examples related to our work are the following. Caro et al. proved in \cite{CHM2019} that for any finite sequence $f:[1,n]\to \{-1,1\}$ satisfying that $\abs{\sum_{i=1}^nf(i)}$ is small, there is a set of consecutive numbers $B\subset [1,n]$ for which $\abs{\sum_{i\in B}f(i)}$ is also small (in particular, small can mean zero-sum). Another interesting work is \cite{BE11}, where Buttkewitz and Elsholtz proved the existence of zero-sum arithmetic progressions with four terms in certain sequences $f:\mathbb N\to \{-1,1\}$.
Balister et al. studied matrices where, for some fixed integer $p$, the sum of each row and each column is a multiple of $p$ \cite{BCRY02}; they showed that these matrices appear in any large enough integer square matrix.

Throughout this paper a matrix with entries in $\{-1,1\}$ will be called a \emph{binary matrix}. Given an $n\times m$ binary matrix $\M=(a_{i,j})$, the \emph{discrepancy} of $\M$ is the sum of all its entries, that is
\begin{equation}\label{eq:disc}
\disc(\M)=\sum_{\substack{1\le i\le n \\ 1\le j\le m}} a_{i,j}.
\end{equation}

Note that if $a^+$ is the number of entries in $\M$ equal to $1$ and $a^-$ is the number of entries in $\M$ equal to $-1$ then
\begin{equation}\label{eq:pm}
\disc(\M) = a^+ - a^- = 2 a^+-nm = nm-2 a^-.
\end{equation}

We define a \emph{zero-sum matrix} $\M$ as a binary matrix with $\disc(\M)=0$.

A \emph{square} $S$ in $M=(a_{i,j})$ is a $2\times 2$ sub-matrix of $\M$ of the form
\[
S =
\begin{pmatrix}
	a_{i,j}   & a_{i,j+s}   \\
	a_{i+s,j} & a_{i+s,j+s}
\end{pmatrix}
\]
for some positive integer $s$. A \emph{zero-sum square} is a square $S$ with $\disc(S)=0$. Note that a square in $\M$ is not zero-sum if and only if it has at least $3$ equal entries.

We are interested in studying matrices $\M$ which do not contain zero-sum squares, we call these matrices \emph{zero-sum-square-free}.

Note that this may also be seen as a $2$-coloring of an $n\times m$ rectangular grid. In this case, zero-sum is the same as balanced.

An $n\times m$ binary matrix $\M=(a_{i,j})$ is called \emph{$t$-split} if for some $0\le t < n+m$,
\[
a_{i,j}=\begin{cases}
	-1 & \text{if } i+j\le t+1, \\
	1  & \text{otherwise. }     \\
\end{cases}\]
If either $\M$, its negative, or its horizontal or vertical reflections are $t$-split for some $t$, we say that $\M$ is \emph{split}. This is relevant since split matrices are always zero-sum-square-free. We are also interested in the possible discrepancies they can have.

\begin{obs}\label{obs:discdiag}
	For an $n\times m$ $t$-split matrix $\M$ with $n\le m$,
	\[\disc(\M)=\begin{cases}
		n m - t (t+1)         & \text{if } t\le n,    \\
		n m+n(n-1)-2 n t      & \text{if } n< t\le m, \\
		(n+m-t-1) (n+m-t)-n m & \text{if } m< t.
	\end{cases}\]
\end{obs}
From this we may conclude the following.
\begin{coro}\label{coro:split}
	Let $\M$ be a $t$-split binary matrix such that $\abs{\disc(\M)} \le n$. If $\M$ is of size $n\times n$, then $t \in \{n-1,n\}$ and $\abs{\disc(\M)} = n$. If $\M$ is of size $n\times (n+1)$, then $t = n$ and $\disc(\M) = 0$.
\end{coro}
In particular, the discrepancy of a square split matrix never vanishes.

Now we are ready to state our main theorem. 
\begin{teo}\label{thm:main}
	Let $n\ge 5$. Every $n\times n$ non-split binary matrix $\M$ with $\abs{\disc(\M)}\le n$ contains a zero-sum square. In particular, every $n\times n$ zero-sum matrix $\M$ contains a zero-sum square.
\end{teo}
Theorem \ref{thm:main} and Corollary \ref{coro:split} immediately yield the following.
\begin{coro}\label{coro:almostsquare}
	Let $\M$ be an $n \times n$ binary matrix. If $n \ge 5$ and $\abs{\disc(\M)} \le n-1$, then $\M$ contains a zero-sum square.
\end{coro}

Our proof method suggests that a stronger result may hold.

\begin{conj}\label{conj}
	For every $C>0$, there is an integer $N$ with the following property: For all $n\geq N$, every $n\times n$ non-split binary matrix $\M$ with $\abs{\disc(\M)}\le Cn$ contains a zero-sum square.
\end{conj}

There is a more general question.
Let $f:\mathbb N \to \mathbb N$ be the function associating to each $n \in \mathbb N$ the largest possible integer $f(n)$ such that every $n \times n$ non-split binary matrix $\M$ satisfying $\abs{\disc(\M)}\le f(n)$ contains a zero-sum square. Obviously $f(n) < n^2$. In fact, $f(n) \le \frac {n^2}{2} + o(n^2)$ as is shown by the $n\times n$ matrix $\M=(a_{i,j})$ defined by
\[
a_{i,j}=\begin{cases}
-1 & \text{if } i,j\text{ are both even,} \\
1  & \text{otherwise.}
\end{cases}
\]
This is a zero-sum-square-free matrix and its discrepancy is about $\frac {n^2}{2}$.
Theorem \ref{thm:main} implies $f(n) \ge n$ if $n \ge 5$. It would be very interesting to determine whether $f(n)$ is linear or quadratic in $n$.

This paper is organized as follows. Section \ref{sec:small} is devoted to particular cases, which were analyzed by a computer. In Section \ref{sec:proof} we give a stronger version of Theorem \ref{thm:main} and its proof. Finally, Section \ref{sec:concl} contains our conclusions and some open questions.

\section{Small cases}\label{sec:small}
Since the proof of Theorem \ref{thm:main} uses induction, we must analyze some of the smaller cases to obtain our induction basis. It is possible to do this by hand but the amount of work is quite large, so we aid ourselves with a computer program.

Our program takes three positive parameters as input: $n$, $m$ and $d$, which should satisfy $n\le m$ and $d\le nm$. The output is a list of all $n\times m$ binary matrices which are zero-sum-square-free and satisfy $\disc(\M)=d$. To do this we use a standard backtracking algorithm that explores all binary matrices with the desired properties.
The code is written in C++ and is available at
\begin{center}
	\url{https://github.com/edyrol/ZeroSumSquares}.
\end{center}

We are mainly interested in two types of zero-sum-square-free matrices: square matrices (with $m=n$) and almost-square matrices (with $m=n+1$). These are the sizes of matrices we need to understand in order to prove Theorem \ref{thm:main}.

Recall that split matrices are zero-sum-square-free, so we always find these examples.

\begin{lema}\label{lem:compu}
	Let $\M$ be a zero-sum-square-free binary matrix with $\abs{\disc(M)}\le 2n$. If $\M$ is of size $n\times (n+1)$ and $4\le n\le 11$, then $\M$ is either a split matrix or it is one of $28$ exceptional $4\times 5$ matrices. If $\M$ is of size $n\times n$ and $5\le n\le 11$, then $\M$ is either a split matrix or it is one of $32$ exceptional $5\times 5$ matrices.
\end{lema}

Although Lemma \ref{lem:compu} mentions $60$ exceptional matrices, there are essentially only $11$. The rest can be obtained by taking the symmetries of these $11$ (generated by reflections and rotations, and their negatives). These $11$ matrices are shown in Figure \ref{fig:small}.

\begin{figure}
	\includegraphics[scale=0.7]{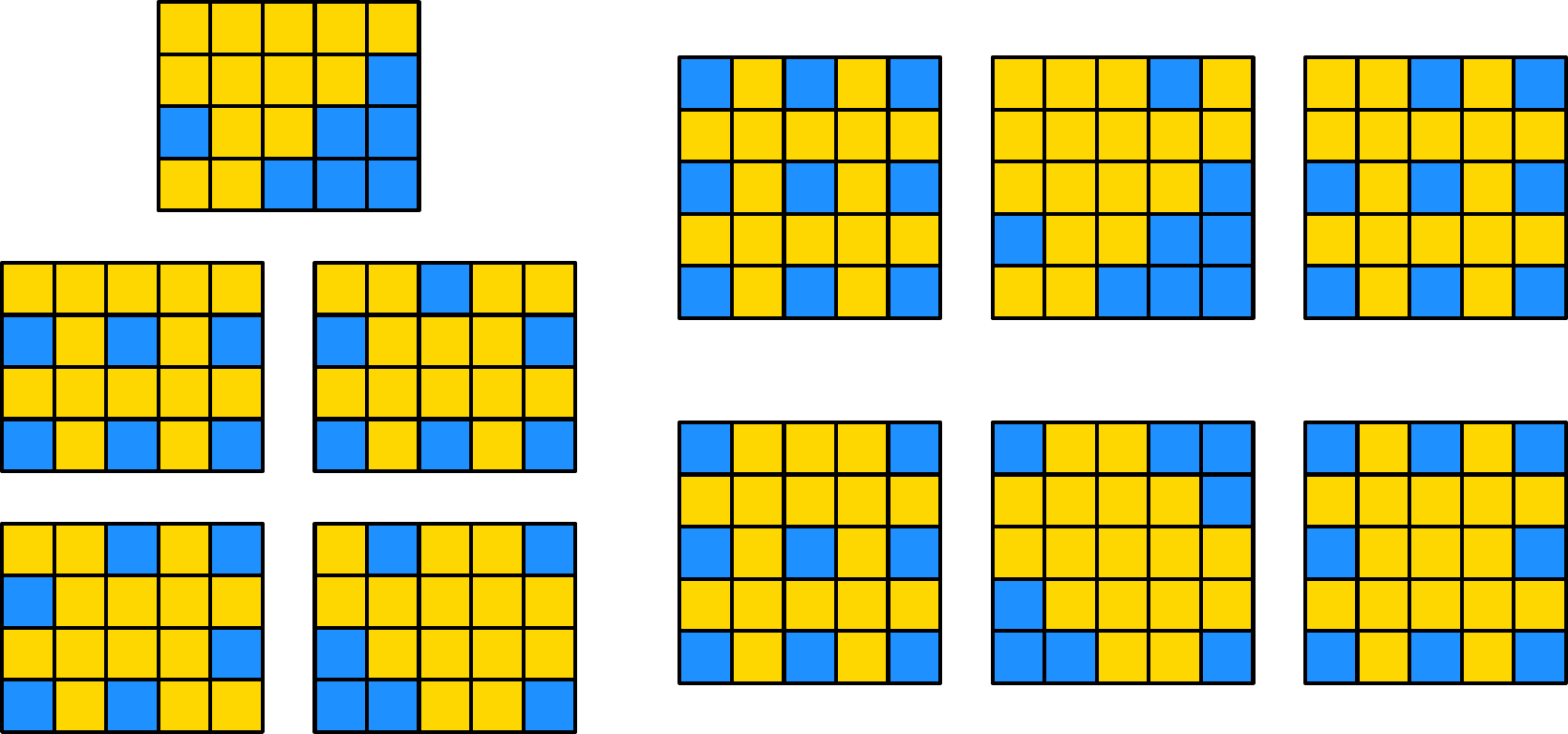}
	\caption{Non-split zero-sum-square-free binary matrices.}
	\label{fig:small}
\end{figure}

The computer program does not take too long to run. Using a home computer with an i7-3770 3.40GHz processor and compiling the program with GCC 8.1.0, it takes less than a second to analyze a $9\times 9$ matrix with fixed discrepancy. For the larger matrices, it can take a couple of minutes. For example, depending on the discrepancy, it takes between $30$ and $50$ seconds to analyze an $11\times 11$ binary matrix and between $1.5$ and $3$ minutes for an $11\times 12$ binary matrix.

\section{Proof}\label{sec:proof}
Our proof of Theorem \ref{thm:main} uses an induction argument. The main idea in the induction step is to split a large zero-sum-square-free matrix $\M$ into four square (with equal side-lengths) or almost-square (with side-lengths differing by $1$) sub-matrices. Since it is not always possible to only use squares, we are forced to understand the behavior of both square and almost-square zero-sum-square-free matrices. For the induction to work, we prove the following stronger statement.

\begin{teo}\label{thm:disc_m}
	Let $n\ge 5$ and $m\in \{n,n+1\}$. Every $n\times m$ non-split binary matrix $\M$ with $\abs{\disc(\M)}\le n$ contains a zero-sum square.
\end{teo}

The basis of the induction is given by the computer analysis described in Section~\ref{sec:small}. It is not indispensable to use a computer to prove Lemma \ref{lem:compu}, although doing it by hand would require either substantial case analysis or a clever argument that has eluded us.

For the rest of the proof we proceed as follows: assuming that the discrepancy of $\M$ is not too large, we find a relatively large sub-matrix $\N$ of $\M$ with small discrepancy. By the induction hypothesis, if we assume that $\M$ is a zero-sum-square-free matrix, we conclude that $\N$ must be split. It turns out that having a relatively large split sub-matrix $\N$ determines the value of many other entries of $\M$. From those values we find that, either $\M$ is itself split as desired, or we can estimate $\disc(\M)$ and find that it is larger than $n$ which contradicts the hypothesis of Theorem \ref{thm:disc_m}.

For integers $h,j,k,l$ satisfying $1\leq h< k\leq n$ and $1\leq j< l\leq m$, we define a \emph{block of $\M$} as the $(k-h+1)\times(l-j+1)$ sub-matrix
\[
	\M[h,k;j,l]=\begin{pmatrix}
	a_{h,j}   & a_{h,j+1}   & \dots  & a_{h,l-1} & a_{h,l}   \\
	a_{h+1,j} & a_{h+1,j+1} & \dots  & a_{h+1,l-1} & a_{h+1,l} \\
	\vdots    & \vdots      & \ddots & \vdots & \vdots    \\
	a_{k-1,j}   & a_{k-1,j+1}   & \dots  & a_{k-1,l-1} & a_{k-1,l}   \\
	a_{k,j}   & a_{k,j+1}   & \dots  & a_{k,l-1} & a_{k,l}
	\end{pmatrix}.
\]

The next lemma shows that a block $\M'$ in a zero-sum-square-free matrix $\M$ is split if a certain sub-block of $\M'$ is also split. It is divided into four instances. Parts (a) and (b) refer to blocks obtained by removing the first row and the first column of $\M'$, respectively. Parts (c) and (d) refer to blocks obtained by removing the last row and the last column of $\M'$, respectively.

\begin{lema}\label{lem:tool}
Let $\M$ be a zero-sum-square-free $n\times m$ matrix with $n\ge 5$. Let $1\leq h< k\leq n$ and $1\leq j< l\leq m$ be integers such that $k-h=l-j=b\geq 2$. Consider the block $\M'=\M[h,k;j,l]$ of size $(b+1)\times(b+1)$,
\begin{enumerate}
\item[(a)] if $\M[h+1,k;j,l]$ is $b$-split then $\M'$ is $(b+1)$-split.
\item[(b)] if $\M[h,k;j+1,l]$ is $b$-split then $\M'$ is $(b+1)$-split.
\item[(c)] if $\M[h,k-1;j,l]$ is $b$-split then $\M'$ is $b$-split.
\item[(d)] if $\M[h,k;j,l-1]$ is $b$-split then $\M'$ is $b$-split.
\end{enumerate}
\end{lema}

\begin{proof}
	If $\M[h+1,k;j,l]$ (respectively $\M[h,k;j+1,l]$) is $b$-split, we need to prove that all entries in the top row (respectively in the leftmost column) of $\M'$ are equal to $-1$. If $\M[h,k;j,l-1]$ (respectively $\M[h,k-1;j,l]$) is $b$-split, we need to prove that all entries in the rightest column (respectively bottom row) of $\M'$ are equal to $1$. Since the arguments are analogous for each case, we only show the first one. Assume that $\M[h+1,k;j,l]$ is $b$-split then, for every $1\leq i \leq b$,
	\begin{equation}\label{eq:values}
	a_{h+i,l-i}=-1 \mbox{ and } a_{h+i,l}=1.
	\end{equation}
	Consider now the square
	\[
		S=\begin{pmatrix}
		a_{h,l-i} & a_{h,l} \\
		a_{h+i,l-i} & a_{h+i,l}
		\end{pmatrix}
	\]
	and recall that, since $\M'$ is a zero-sum-square-free matrix, any square $S$ in $\M_0$ has at least $3$ equal entries. Thus, (\ref{eq:values}) implies that $a_{h,l-i}=a_{h,l}$ for every $1\leq i \leq b$. Therefore, the elements in the first row of $\M'$, $a_{h,j},\dots,a_{h,l}$, are all equal. Finally, since $a_{h,j}=a_{h,j+1}$ and $a_{h+1,j}=a_{h+1,j+1}=-1$, the same argument for the square
	\[S=\begin{pmatrix}
		a_{h,j} & a_{h,j+1} \\
		a_{h+1,j} & a_{h+1,j+1}
	\end{pmatrix}\]
	implies that $a_{h,j}=a_{h,j+1}=-1$, so all entries in the top row of $\M'$ are equal to $-1$. This shows that $\M'$ is indeed $(b+1)$-split.
\end{proof}

Once we have a $t$-split block $\M'$, we can also deduce the values of other entries which are not necessarily adjacent to $\M'$.

\begin{lema}\label{lem:tool2}
Let $\M$ be a zero-sum-square-free $n\times m$ matrix, where $\M' = \M\left[1,k;1,l\right]$ is $t$-split with $t < k <n$ and $t < l <m$.

If $l < r \le \min(t+l-1,n)$, then the entries $a_{r,i}$ have the same value for
\[i\in [1,\floor{(t+l-r+1)/2}]\cup [r-t+1,l].\]

Analogously, if $k<c\le \min(t+k-1,m)$, then the entries $a_{i,c}$ have the same value for
\[i\in [1,\floor{(t+k-c+1)/2}]\cup [c-t+1,k].\]
\end{lema}
\begin{proof}
	Assume $1\le i\le (t+l-r+1)/2$ and consider the square
	\[
	S = \begin{pmatrix}
	a_{r-l+i,i} & a_{r-l+i,l} \\
	a_{r,i} & a_{r,l}
	\end{pmatrix}.
	\]
	Note that, since $\M'$ is $t$-split and $r-l+2i\le t+1$, $a_{r-l+i,i}=-1$ and $a_{r-l+i,l}=1$.
	So two entries of $S$ have opposite values and therefore $a_{r,i}=a_{r,l}$.
	
	If $r-t+1 \le i \le l$, consider the square
	\[
	S = \begin{pmatrix}
	a_{r+1-i,1} & a_{r+1-i,i} \\
	a_{r,1} & a_{r,i}
	\end{pmatrix}.
	\]
	Since $\M'$ is $t$-split and $r+1-i\le t$, $a_{r+1-i,1}=-1$. Furthermore, since $i\le l$, $a_{r+1-i,i}=1$. So two entries of $S$ have opposite values and therefore $a_{r,1}=a_{r,i}$.
	
	In conclusion, $a_{r,i}=a_{r,l}$ for any $i\in [1,(t+l-r+1)/2]$, in particular, $a_{r,1}=a_{r,l}$. If $i\in [r-t+1,l]$ then $a_{r,i}=a_{r,1}$. Therefore, all of these values are equal. The proof for columns is analogous.
\end{proof}

\begin{proof}[Proof of Theorem \ref{thm:disc_m}]
By Lemma \ref{lem:compu} we know the theorem holds for any $n\le 11$ and $m\in\{n,n+1\}$.
Let $\M=(a_{i,j})$ be a $n\times m$ binary matrix with $n\ge 12$, $m\in\{n,n+1\}$ and $\abs{\disc(\M)}\le n$. We need to prove that either $\M$ is split or it contains a zero-sum square, so we assume henceforth that $\M$ is zero-sum-square-free.

As stated before, we use induction on $n$, so we may assume that the theorem holds true for all square and almost-square binary matrices with smaller dimensions than those of $\M$.

We consider the four blocks of $\M$ formed by splitting $\M$ vertically and horizontally as evenly as possible. To be precise, let
\begin{alignat*}{2}
\M_1 & =\M\Big[1,\floor{\frac n2};   &  & \quad 1,\floor{\frac m2}\Big],             \\
\M_2 & =\M\Big[\floor{\frac n2}+1,n; &  & \quad1,\floor{\frac m2}\Big],              \\
\M_3 & =\M\Big[1,\floor{\frac n2};   &  & \quad\floor{\frac m2}+1,m\Big] \text{ and} \\
\M_4 & =\M\Big[\floor{\frac n2}+1,n; &  & \quad\floor{\frac m2}+1,m\Big].
\end{alignat*}
Note that, for $1\leq i\leq 4$, each block $\M_i$ is either a square or an almost-square matrix. Also, the smallest side of any $\M_i$ is $\floor{\frac n2}$ and the largest is $m-\floor{\frac m2}=\ceil{\frac m2}\le\ceil{\frac{n+1}2}\le \floor{\frac n2}+1$. Therefore, the side-lengths of each $\M_i$ are in the set $\{\floor{\frac n2},\floor{\frac n2}+1\}$.

\begin{claim}\label{claim:4split}
	Either one of the four matrices $\M_i$ satisfies $\abs{\disc{(\M_i)}}<\floor{\frac n2}$ or two of these four matrices have discrepancies with opposite signs.
\end{claim}
\begin{proof}
	If this is not the case and that the four matrices satisfy $\disc{(\M_i)}\ge\floor{\frac n2}$, then $n\ge \disc(\M)=\sum\disc(\M_i)\ge 4\floor{\frac n2}$ which is a contradiction. If the four matrices satisfy $\disc{(\M_i)}\leq-\floor{\frac n2}$ we obtain a contradiction in the same way. Therefore two of the $\M_i$ have discrepancies with opposite signs.
\end{proof}
	
What we actually wish to find is a relatively large block of $\M$ with small discrepancy. By an interpolation argument this is easily achievable.
\begin{claim}\label{claim:matrixN}
	By exchanging $1$ and $-1$ if necessary, we may assume that there is an almost-square $\floor{\frac n2}$-split block $\N$ with side-lengths in the set $\{\floor{\frac n2},\floor{\frac n2}+1\}$ such that $\abs{\disc(\N)}<\floor{\frac n2}$.
\end{claim}
\begin{proof}
	Claim \ref{claim:4split} either provides the block we want or it gives us two blocks $\N_+$ and $\N_-$ from the set $\{\mathcal{M}_1,\mathcal{M}_2,\mathcal{M}_3,\mathcal{M}_4\}$ with $\disc(\N_+)>0$ and $\disc(\N_-)<0$.
		
	We can construct a sequence $\N_-=\N_1,\N_2,\dots,\N_k=\N_+$ of blocks of $\M$ with the following properties:
	\begin{itemize}
		\item The side-lengths of every $\N_i$ are in $\{\floor{\frac n2},\floor{\frac n2}+1\}$.
		\item For each $1\le i<k$, one of $\N_i$ and $\N_{i+1}$ can be obtained from the other by removing one row or one column.
	\end{itemize}
	In other words, we start with $\N_1=\N_-$ and start moving towards $\N_+$. In each step we add or remove a row or column to $\N_i$ taking care to always leave $\N_{i+1}$ with side-lengths in the set $\{\floor{\frac n2},\floor{\frac n2}+1\}$. Note that in each step we switch from square to almost-square and vice-versa.
		
	At some point the discrepancy changes from negative to positive, so assume that $\disc(\N_i)<0$ and $\disc(\N_{i+1})>0$ for some $1\le i<k$. Since, at each step the discrepancy changes by at most $\floor{\frac n2}+1$, we conclude that either $\N_i$ or $\N_{i+1}$ must have absolute discrepancy at most $(\floor{\frac n2}+1)/2<\floor{\frac n2}$. Let $\N$ be this block.
	
	Now we can use our induction hypothesis on $\N$. Since $\abs{\disc(\N)}<\floor{\frac n2}$, either $\N$ contains a zero-sum square, or it must necessarily be split. Furthermore, by Corollary \ref{coro:split}, if $\N$ is zero-sum-square-free, then it must be an almost-square block and have discrepancy exactly $0$.
\end{proof}
	
In the following Claim we prove that several entries of $\M$ are forced. Note that, if $\N=\M[p,r;q,s]$, the $\floor{\frac n2}$-th diagonal of $\N$ is contained in the $(p+q+\floor{\frac n2}-2)$-th diagonal of $\M$. So, to simplify things, we define
\begin{equation}\label{eq:lower_t}
	t=p+q+\floor{\frac n2}-2\ge\floor{\frac n2}.
\end{equation}

\begin{claim}\label{claim:tsplit}
	By relabeling the entries of $\M$ and exchanging $1$ and $-1$ if necessary, we may assume that the block $\M_0 = \M\left[1,t+1;1,t+1\right]$ is $t$-split.
\end{claim}
\begin{proof}
	We start with the block $\N=\M[p,r;q,s]$ described in Claim \ref{claim:matrixN}. We repeatedly apply Lemma~\ref{lem:tool} to obtain a sequence of split matrices $\N=\N_1\dots,\N_k$ in the following way.
	Assume that $\N_i=\M[h,k;j,l]$ is a $b$-split block. The block $\N_i$ is either square or almost-square and $b$ differs from the side-lengths of $\N_i$ by at most $1$. There are four possibilities.
	\begin{itemize}
		\item If $b=k-h+1=j-l$, then it follows from parts (a) and (c) of Lemma \ref{lem:tool} that $\M[h-1,k;j,l]$ is $(b+1)$-split (if $h>1$) and $\M[h,k+1;j,l]$ is $b$-split (if $k<n$).
		\item If $b=k-h=j-l+1$, then parts (b) and (d) of Lemma \ref{lem:tool} imply that $\M[h,k;j-1,l]$ is $(b+1)$-split (if $j>1$) and $\M[h,k;j,l+1]$ is $b$-split (if $l<m$).
		\item If $b=k-h=j-l$ then we can remove the last row or column from $\N$ and apply parts (c) and (d) of Lemma \ref{lem:tool} to show that $\M[h-1,k;j,l]$ (if $h>1$) and $\M[h,k;j-1,l]$ (if $j>1$) are $b$-split.
		\item If $b=k-h+1=j-l+1$ then we can remove the first row or column from $\N$ and apply parts (a) and (b) of Lemma \ref{lem:tool} to show that $\M[h,k+1;j,l]$ (if $k<n$) and $\M[h,k;j,l+1]$ (if $l<m$) are $b$-split.
	\end{itemize}
	In any case, let $\N_{i+1}$ be any of the larger split blocks described above, whenever possible.
	
	The process can only stop at $\N_k=\M[h,k;j,l]$ if either $\N_k=\M$ or $\N_k$ is square with $(h,j)=(1,1)$ or $(k,l)=(n,m)$.
	If $(h,j)=(1,1)$, we are done. If $(k,l)=(n,m)$ then we may relabel the entries of $\M$, exchanging $(1,1)$ and $(n,m)$ and exchange $1$ and $-1$ to obtain the desired result.
\end{proof}

\begin{claim}\label{claim:tsplit2}
	The block
	\[
		\M_1 = \M\left[1,\min\left(n,\floor{\frac {3t}2}\right);1,\min\left(m,\floor{\frac {3t}2}\right)\right]
	\]
	is $t$-split.
\end{claim}
\begin{proof}
	We start with the block $\M_0 = \M\left[1,t+1;1,t+1\right]$ from Claim \ref{claim:tsplit} and repeatedly apply Lemma \ref{lem:tool2} in the following way.
	
	If $\M'=\M[1,k;1,k]$ is $t$-split with $t+1\le k<m$, apply Lemma \ref{lem:tool2} for columns with $l=k$ and $c=k+1$. The values $a_{i,c}$ are all equal for $i\in [1,k]$ whenever
	\[(c-t+1)-1 \le (t+k-c+1)/2,\]
	which is equivalent to $k+1=c\le 3t/2$.
	If this is the case, consider the square
	\[\begin{pmatrix}
	a_{k-1,k}& a_{k-1,k+1}\\
	a_{k,k}& a_{k,k+1}
	\end{pmatrix}.\]
	Since $a_{k-1,k}=a_{k,k}=1$ and $a_{k-1,k+1}=a_{k,k+1}$, then $a_{k,k+1}$ and therefore every $a_{i,k+1}$ with $i\in[1,k]$ is $1$.
	Thus, $\M[1,k;1,k+1]$ is $t$-split as long as $k+1\le\floor{3t/2}$.
	
	Now, starting with $\M[1,k;1,k+1]$, apply Lemma \ref{lem:tool2} for rows with $l=k+1$ and $r=k+1$. The values $a_{r,i}$ are all equal for $i\in [1,l]$ whenever
	\[(r-t+1)-1 \le (t+l-r+1)/2,\]
	which is equivalent to $k+1=r\le 3t/2+1/2$.
	In the same way as before, we may conclude that $a_{i,k+1}=1$ if $i\in[1,k+1]$.
	Thus, $\M[1,k+1;1,k+1]$ is $t$-split whenever $k+1\le\floor{3t/2}$.
	
	This process stops when either $r$ or $c$ exceeds $\floor{3t/2}$ or the corresponding dimension of $\M$.
\end{proof}

In view of the previous Claim we may assume that $m\ge \floor{\frac{3t}{2}}+1$, otherwise $\M=\M_1$ is a split matrix. Since $m\le n+1$, this implies that
\begin{equation}\label{eq:upper_t}
	t\le\frac{2n+1}{3},
\end{equation}
which will be relevant later.
In the case in which $\M_1$ does not cover $\M$, we may infer the values of additional entries of $\M$. This is done in a similar way to Claim \ref{claim:tsplit2}, although we are no longer able to obtain a $t$-split matrix. Instead, we obtain five regions outside of $\M_1$ for which $a_{i,j}=1$. These are illustrated in Figure \ref{fig:ext1}.

\begin{figure}
	\includegraphics[scale=0.7]{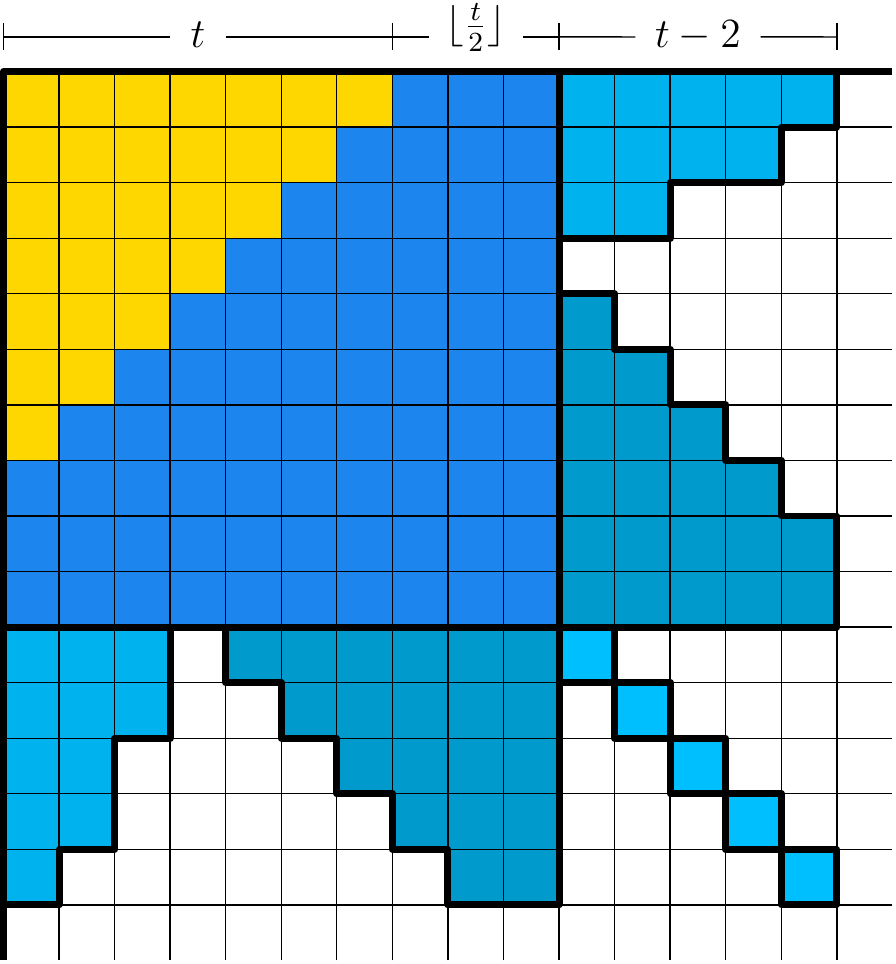}
	\caption{The regions described in Claims \ref{claim:tsplit}, \ref{claim:tsplit2} and \ref{claim:ones}. The colors yellow and blue represent values of $-1$ and $1$, respectively. The matrix $\M$ is not actually large enough to contain all of the marked squares.}
	\label{fig:ext1}
\end{figure}

The first region has a triangular shape bounded by the first column outside of $\M_1$, the first row of $\M$ and a certain line of slope $\frac 12$.
The second region is bounded by by the first column outside of $\M_1$, the last row of $\M_1$ and a line of slope $-1$.
Of course, there are corresponding regions to these below $\M_1$.
Lastly, the entries of the diagonal of $\M$ which are outside of $\M_1$ must also have value $1$.
This is formalized in the following claim.

\begin{claim}\label{claim:ones}
	Let $T=\floor{\frac {3t}2}$, then $a_{i,j}=1$ and $a_{j,i}=1$ whenever $T < j$ and any of the following hold:
	\begin{enumerate}
		\item[(a)] $i \le \floor{\frac{T+t+1-j}{2}}$,
		\item[(b)] $j-t< i\le T$, or
		\item[(c)] $i=j$.
	\end{enumerate}
\end{claim}
\begin{proof}
We start with the $t$-split block $\M_1 = \M\left[1,T;1,T\right]$ from Claim \ref{claim:tsplit2}.

We inductively deduce the values in column $j$ starting with $j=T+1$ and increasing $j$ one by one. Take $k=l=T$ and $c=j$ in Lemma \ref{lem:tool2} for columns.

Note that \eqref{eq:lower_t} implies that $j \le T+t-2$, so two things happen;
the values $a_{i,j}$ are all equal for
\[i\in [1,\floor{(t+T-j+1)/2}]\cup [j-t+1,T]\]
and the interval $[j-t+1,T]$ contains at least two elements.
By considering the square
\[\begin{pmatrix}
a_{T-1,j-1}& a_{T-1,j}\\
a_{T,j-1}& a_{T,j}
\end{pmatrix}\]
and using the fact that the elements $a_{T-1,j-1}$ and $a_{T,j-1}$ from the previous column have value $1$, we conclude that all the $a_{i,j}$ described above are equal to $1$.

Analogously, using Lemma \ref{lem:tool2} for rows, we can say the same for $a_{j,i}$.

Condition $i\in [1,\floor{(t+T-j+1)/2}]$ is equivalent to $i \le \floor{(T+t+1-j)/2}$ which proves part (a) of the claim, while $i\in [j-t+1,T]$ is equivalent to $j-t+1\le i \le T$ which proves part (b).

To prove part (c), for $T<i=j\le T+t-2$, consider the square
\[\begin{pmatrix}
a_{1,1}& a_{1,i}\\
a_{i,1}& a_{i,i}
\end{pmatrix}.\]
Since $a_{1,1}=-1$ and $a_{i,1}=a_{1,i}=1$, we must have that $a_{i,i}=1$.
\end{proof}

Now we can bound the discrepancy of $\M$. Recall from \eqref{eq:pm} that it is enough to know the number of positive entries $a^+$ of $\M$ in order to compute $\disc(\M)$.
Since $2a^+-nm=\disc(\M)\le n$, we have that
\begin{equation}\label{eq:onesle}
a^+\le \frac{n+nm}{2}.
\end{equation}
If this equation is violated, it means that $\M$ is not larger than $\M_1$. So, all that remains is to bound from below the number of positive entries $a^+$ of $\M$.

Let $R=n-T$ and define
\begin{align*}
	a_0 &= \frac{t(t-1)}{2}+\floor{\frac t2}^2+2 t\floor{\frac t2},\\
	a_1 &= 2\sum_{j=T+1}^n \floor{\frac {T+t+1-j}{2}}\\
	&=2\sum_{k=1}^R \floor{\frac {t+1-k}{2}},\\
	a_2 &= 2\sum_{j=T+1}^n (T-j+t)\\
	&= 2 \sum_{k=1}^R (t-k) \text{ and}\\
	a_3 &= R.
\end{align*}

A simple calculation gives the following claim.
\begin{claim}\label{claim:entries1}
	\[a^+\ge a_0+a_1+a_2+a_3\]
\end{claim}
\begin{proof}
	The number of positive entries in $\M_1$, described in Claim \ref{claim:tsplit2}, is $a_0$.
	
	If $m=n$, then $a_1$, $a_2$ and $a_3$ equal the number of positive entries described in parts (a), (b) and (c) of Claim \ref{claim:ones}, respectively.
	
	If $m=n+1$, by ignoring the positive entries in the last column of $\M$, we obtain that $a_1$, $a_2$ and $a_3$ are lower bounds for the number of positive entries described in parts (a), (b) and (c) of Claim \ref{claim:ones}, respectively.
\end{proof}

Recall that we are currently dealing with $n\ge 12$ and, from \eqref{eq:lower_t} and \eqref{eq:upper_t},
\begin{equation}\label{eq:tinterval}
\floor{\frac{n}{2}} \le t \le \floor{\frac{2n+1}{3}}.
\end{equation}
Before simplifying this lower bound, we can check that \eqref{eq:onesle} cannot be satisfied for small values of $n$. The following claim can be easily verified with aid from a computer.

\begin{claim}\label{claim:small}
	For $12\le n\le 15$ and $\floor{\frac n2}\le t\le \floor{\frac {2n+1}{3}}$, we have that
	\[a_0+a_1+a_2+a_3 > \frac {(n^2+2n)}{2}\ge \frac {(n+nm)}{2}.\]
\end{claim}

Therefore, we may assume that $n\ge 16$. What follows is a series of algebraic manipulations to obtain a simpler lower bound for $a^+$ which can be analyzed analytically.

\begin{claim}
	For $n\ge 16$ and $\floor{\frac n2}\le t\le \floor{\frac {2n+1}{3}}$, we have that
	\[a_0+a_1+a_2+a_3 \ge \frac{23 n^2-70 n-77}{32}.\]
\end{claim}

\begin{proof}
We can remove the integer parts in $a_0+a_1+a_2+a_3$ by using that, for any integer $x$, $\frac{x-1}2\le \floor{\frac x2}\le \frac x2$. It is convenient to do this in two parts, first we apply these inequalities but leave the variable $r$ as it is. This gives
\begin{align*}
	a_0+a_1+a_2+a_3 &= \frac{t(t-1)}{2}+\floor{\frac t2}^2+2 t\floor{\frac t2} + 2\sum_{k=1}^R \left(\floor{\frac {t+1-k}{2}}\right)\\
	&\quad + 2 \sum_{k=1}^R (t-k) + R\\
	&\ge \frac{t(t-1)}{2}+\left(\frac{t-1}2\right)^2+t(t-1) + 2\sum_{k=1}^R\left(\frac{t-k}{2}\right)\\
	&\quad + 2\sum_{k=1}^R\left(t-k\right) +R\\
	&= \frac{7 t^2}{4} - 2 t + \frac{1}{4}+3 R t-\frac{3}{2} R^2-\frac{1}{2}R.
\end{align*}
Since $R=n-\floor{\frac {3t}2}$ we have that $n-\frac{3t}2\le R\le n-\frac{3t-1}2$, using this on the last expression we obtain
\begin{align*}
	a_0+a_1+a_2+a_3 &\ge \frac{7 t^2}{4} - 2 t + \frac{1}{4}+3 \left(n-\frac{3t}2\right)t\\
	&\quad -\frac{3}{2} \left(n-\frac{3t-1}2\right)^2-\frac{1}{2}\left(n-\frac{3t-1}2\right)\\
	&= -\frac{49 t^2}{8}+\frac{15 n t}{2}+t-\frac{3 n^2}{2}-2 n-\frac{3}{8}.
\end{align*}
To minimize this last expression think of $n$ as fixed and consider it as a function of $t$. Then this is an upside-down parabola and, from \eqref{eq:tinterval}, the relevant values for $t$ are contained in the interval $\left[\frac{n-1}{2},\frac{2n+1}{3}\right]$. Therefore the parabola is bounded from below by the minimum between the values at $t=\frac{n-1}{2}$ and $t=\frac{2n+1}{3}$. These are, respectively,
\[
	\frac{23 n^2-70 n-77}{32}\qquad \text{and} \qquad\frac{14 n^4-28 n -13}{18}.
\]
The former gives the smallest value.
\end{proof}

To conclude the proof, notice that the parabolas $\frac{1}{32}(23 n^2-70 n-77)$ and $\frac 12 (n^2+2n)$ intersect twice, once in the interval $(-1,0)$ and a second time in the interval $(15,16)$. Since $n\ge 16$, we have
\[
	a^+ \ge a_0+a_1+a_2+a_3 \ge \frac{23 n^2-70 n-77}{32} > \frac{n^2+2n}{2}.
\]
This contradicts \eqref{eq:onesle}, so $\M=\M_1$ and therefore $\M$ is a split matrix.
\end{proof}

\section{Conclusions and further work}\label{sec:concl}
We were able to give an elemental proof of Theorem \ref{thm:main}, but we are sure that there is a deeper result in the direction of Conjecture \ref{conj}. It is also likely that something can be said for non-square matrices.
The fact that the final bound given for $a^+$ is significantly smaller than $n^2$ suggests that a much stronger theorem should hold.
It is possible to strengthen our proof to obtain a stronger version of Theorem \ref{thm:disc_m} with something like $\abs{\disc(\M)}\le 2n$ instead of $\abs{\disc(\M)}\le n$, however significantly more work is required to establish this and it is probably not worth the effort.

In \cite{RBM2010} Erickson matrices were generalized to \emph{$3$-squares}. A $k$-square in a matrix $\M$ is a $k\times k$ sub-matrix of $\M$ contained in $k$ rows of $\M$ of the form $i,i+s,\dots,i+(k-1)s$ and $k$ columns of $\M$ of the form $j,j+s,\dots,j+(k-1)s$.
We could ask about zero-sum-$k$-square-free binary matrices but this does not make sense when $k$ is odd. However, the case when $k$ is even seems interesting. For odd $k$ we can ask about binary matrices which don't have $k$-squares of sum $\pm 1$.

Lastly, we should point out that with the aid of Claims \ref{claim:tsplit} and \ref{claim:ones}, or with stronger versions of this claim, zero-sum-square-free matrices of much larger sizes may be analyzed by a computer. This might be useful for generalizing our results. However, a different type of computer search might likely be much more useful.
SAT-solvers have been used for finding lower bounds in Ramsey-like problems (see e.g. \cite{HHLM2007}) but it is not obvious how to include the discrepancy condition here. Perhaps linear integer programming could work.
Since we didn't need to analyze anything larger than an $11\times 12$ matrix, we didn't work much on making our program efficient.

\section*{Acknowledgments}
The authors would like to thank the anonymous referee for his comments which improved the paper greatly. They are also thankful for the facilities provided by the Banff International Research Station ``Casa Matemática Oaxaca'' during the ``Zero-Sum Ramsey Theory: Graphs, Sequences and More'' workshop (19w5132).
This research was supported by CONACyT project 282280 and PAPIIT project IN116519.


\begin{thebibliography}{HHLM07}
	
	\bibitem[AM08]{AM2008}
	M.~Axenovich and J.~Manske, \emph{On monochromatic subsets of a rectangular
		grid}, Integers \textbf{8} (2008), A21, 14.
	
	\bibitem[BE10]{BE2010}
	R.~Bacher and S.~Eliahou, \emph{Extremal binary matrices without constant
		2-squares}, J. Comb. \textbf{1} (2010), no.~1, 77--100.
	
	\bibitem[BCRY02]{BCRY02}
	P. Balister, Y. Caro, C. Rousseau and R. Yuster, \emph{Zero-sum square matrices},
	Eur. J. Combin., \textbf{23} (2002), no. 5, 489--497.
	
	\bibitem[BE11]{BE11}
	Y. Buttkewitz and C. Elsholtz, \emph{Patterns and complexity of multiplicative functions},
	J. London Math. Soc. (2) \textbf{84} (2011), no. 3, 578--594.
		
	\bibitem[Cha01]{Cha2001}
	B.~Chazelle, \emph{The discrepancy method: randomness and complexity},
	Cambridge University Press, 2001.
	
	\bibitem[CHM19]{CHM2019}
	Y.~Caro, A.~Hansberg, and A.~Montejano, \emph{Zero-sum subsequences in
		bounded-sum {$\{-1,1\}$}-sequences}, J. Combin. Theory Ser. A \textbf{161}
	(2019), 387--419.
	
	\bibitem[Eri96]{Eri1996}
	M.~J. Erickson, \emph{Introduction to combinatorics}, Wiley-Interscience Series
	in Discrete Mathematics and Optimization, John Wiley \& Sons, Inc., New York,
	1996, A Wiley-Interscience Publication.
	
	\bibitem[HHLM07]{HHLM2007}
	P.~R. Herwig, M.~J.~H. Heule, P.~M. {van Lambalgen}, and H.~{van Maaren},
	\emph{A new method to construct lower bounds for van der {W}aerden numbers},
	Electron. J. Combin. \textbf{14} (2007), no.~1, Research Paper 6, 18.
	
	\bibitem[RBM10]{RBM2010}
	D.~Robilliard, A.~Boumaza, and V.~M{arion-Poty}, \emph{Meta-heuristic search
		and square {E}rickson matrices}, IEEE Congress on Evolutionary Computation,
	IEEE, 2010, pp.~1--8.
	
	\bibitem[Tao16]{Tao2016}
	T.~Tao, \emph{The {E}rdős discrepancy problem}, Discrete Anal. (2016),
	Paper No. 1, 29.
	
\end{thebibliography}
\end{document}